\documentclass[a4paper,twoside,12pt]{article}      
\usepackage{amsmath,amssymb,amsfonts} 
\usepackage{amsthm}
\usepackage{graphics}                 
\usepackage{color}                    
\usepackage{hyperref}                 
\usepackage{graphicx,eucal}
\usepackage{latexsym,epsfig,epic,eepic}

\newtheorem{theorem}{Theorem}[section]

\newtheorem{lemma}[theorem]{Lemma}

\theoremstyle{definition}

\date{2011}
\title{The Poisson boundary of a locally discrete group 
of diffeomorphisms of the circle}

\author{Bertrand Deroin}

\begin{document}
\maketitle
\begin{abstract} 
We compute the Poisson boundary of locally discrete groups of diffeomorphisms of the circle. 
\end{abstract}

\begin{small}

\section{Introduction} 

In this work, we study \textit{non elementary} subgroups of the group of diffeomorphisms of the circle ${\bf S}^1 = {\bf R} / {\bf Z}$, i.e. acting on the circle without preserving any probability measure. For such a group, there is a unique invariant compact set in which every orbit is dense, called the \textit{limit set}. Either this set is the whole circle, or it is a Cantor set, see for instance~\cite{Ghys}. Our goal is to understand some aspects of the dynamics of a non elementary subgroup of the group of diffeomorphisms of the circle on its limit set, in the case that the group has some discreteness properties. 

Given two integers $l\leq k$ (with $k$ possibly $\infty,\omega$), a non elementary subgroup $G$ of the group of $C^k$-diffeomorphisms of the circle 
is called \textit{locally discrete} (resp. \textit{strongly locally discrete}) in the $C^l$-topology if there is a covering of the limit set  
by intervals $I_i$ such that if a sequence of elements $g_n$ of $G$ 
converges to the identity in the $C^l$-topology on one of the $I_i$'s, then the restriction of $g_n$ to $I_i$ (resp. $g_n$) 
is identically identity when $n$ is large enough. If $k=\omega$, then the $C^{\omega}$-topology is defined in the sense of Ghys, see~\cite{ghys2}: namely, a sequence of diffeomorphisms $g_n$ of the circle converge to the identity in the $C^{\omega}$-topology if there 
exists a compact neighborhood $V$ of ${\bf R}/{\bf Z}$ in the annulus ${\bf C}/{\bf Z}$, such that every $g_n$ extends as an injective holomorphic map $\widetilde{g_n}: V \rightarrow {\bf C} \backslash {\bf Z}$, which converges to the identity uniformly on $V$ when $n$ tends to infinity.

As an example, a non elementary subgroup of $\mathrm{Diff}^{\omega}( {\bf S}^1)$ whose limit set is a Cantor set is automatically locally discrete in the $C^{\omega}$-topology. In fact if a non elementary subgroup of analytic diffeomorphisms of the circle is not locally discrete in the $C^{\omega}$-topology, then the action is
\begin{itemize} 
\item \textit{minimal}, i.e. every orbit is dense, 
\item \textit{ergodic} with respect to the  
Lebesgue measure, i.e. every $G$-invariant Borel set has Lebesgue measure $0$ or $1$, and
\item \textit{measurably rigid}, namely every action $G\rightarrow \mathrm{Diff}^{\omega} ({\bf S^1})$ which is measurably conjugated to the representation $\mathrm{id}_G$ is in fact conjugated to it by an analytic diffeomorphism, see~\cite{Rebelo}.
\end{itemize} 
These properties follows from the fact that there exist local analytic flows which are limits of elements of the group in the $C^{\omega}$-topology, 
see the combination of the works~\cite{ghys2,nakai,LR}, and~\cite{DKN} for the existence of hyperbolic fixed points. It is an interesting problem to study the groups of diffeomorphisms that do not satisfy one of these properties, and to try to classify them up to some equivalence relations.  




Our purpose here is understand the action of a strongly locally discrete group of diffeomorphisms on its limit set, up to measurable conjugacy. For instance, suppose that $G$ is a finitely generated subgroup of the group of diffeomorphisms of the circle which is strongly locally discrete and Gromov hyperbolic. Let $\partial G$ be the geometric boundary of $G$, $\mu$ a probability measure on $G$ and $\nu$ the stationary measure on $\partial G$, i.e. the probability measure $\nu$ such that $\int _G g \nu\ d\mu(g) = \nu$. If we denote by $\nu'$ the stationary measure on the circle, then we prove that there exists a $G$-equivariant measurable isomorphism between $(\partial G, \nu)$ and 
a finite $G$-equivariant quotient of $({\bf S}^1,\nu')$.  

The existence of a measurable isomorphism is the consequence of a statement valid for every pair $(G, \mu)$, without assuming $G$ to be Gromov hyperbolic. Namely, we prove that if $G$ is a non elementary strongly locally discrete subgroup of the group of  diffeomorphisms of the circle, then the Poisson boundary of the pair $(G, \mu)$ is a finite quotient of the circle, see theorem~\ref{T: poisson boundary} for the precise regularity assumptions. Then, the previously announced statement is a consequence of a theorem by Kaimanovich~\cite{Kaimanovich} stating that a model for the Poisson boundary of a Gromov hyperbolic group is indeed its geometric boundary.

The proof is based on our work~\cite{DKN} in collaboration with Kleptsyn and Navas, showing that if a group of homeomorphisms of the circle is non elementary, then a finite quotient of the circle is a boundary for the $\mu$-random walk on $G$, see section~\ref{ss: the circle as a boundary}. The argument then follows essentially the same strategy as the one used by Ledrappier in~\cite{L2} in the case of a lattice contained in $\mathrm{PSL}(2,{\bf C})$, using the entropy criterion of Derriennic and Kaimanovich/Vershik, see~\cite{KV, Derriennic}. The difficulty is then to control the affine/projective distortions of the random compositions, depending of the degree of discreteness that we have at hand.



\subsection{The circle as a probabilistic boundary} \label{ss: the circle as a boundary}

Let $G$ be a non elementary countable subgroup of the group of homeomorphisms of the circle ${\bf S}^1 = {\bf R}/{\bf Z}$. Recall that $G$ is called non elementary if it does not preserve a probability measure on ${\bf S}^1$. 
Let $\mu$ be a probability measure on $G$ whose support generates $G$ as a semi-group, and let $\nu$ 
be a stationary measure on the circle, that is, a probability measure such that 
$\int g \nu\ d\mu(g) = \nu$. Such a measure is unique, see~\cite[Proposition 5.5]{DKN}.  

Consider the right random walk on $G$ defined by 
\[ {\bf r_n} = g_1\ldots g_n,\]
where $g_n$ is a sequence of $G$-valued independant random variables with distribution $\mu$. 
For instance $g_n$ is the $n$-th coordinate of the probability space $(G^{\bf N},\mu^{\bf N})$.  
Furstenberg observed~\cite{Furstenberg} that for $\mu^{\bf N}$-almost every ${\bf g}\in G^{\bf N}$ the sequence 
${\bf r_n} \nu$ converges weakly to a probability measure $\nu_{\bf g}$. 
If almost surely the measure $\nu_{\bf g}$ is a Dirac mass, then the couple 
$({\bf S}^1, \nu)$ is called a $(G,\mu)$\textit{-boundary}. 

In this section we review our work~\cite{DKN} in collaboration with Kleptsyn and Navas, where it is shown that a finite equivariant quotient of $({\bf S} ^1 , \nu)$ is a boundary of the pair $(G,\mu)$. 

A well-known condition to ensure that the couple $({\bf S}^1,\nu)$ is a $(G,\mu)$-boundary is that the action is \textit{proximal}, that is, every closed interval can be mapped by elements of $G$ to intervals of arbitrary small length. This criterion was proved the first time by Antonov~\cite{Antonov} (see also the work~\cite{KleptsynNalski} by Kleptsyn-Nalski). The action of a non elementary subgroup of $\mathrm{Homeo}({\bf S}^1)$ is not always proximal. For instance, considering a finite order rotation, the group consisting of the elements which commutes with this rotation is non elementary, but it is not proximal. Another example is obtained by the classical Denjoy's procedure, which consists in blowing-up an orbit (or a countable number of orbits) of a non elementary subgroup of $\mathrm{Homeo}({\bf S}^1)$, i.e. replacing every point of this orbit by an interval. The new action is semi-conjugated to the initial one, but it is no more proximal. Ghys showed that essentially, these are the only examples of non elementary subgroups of $\mathrm{Homeo}({\bf S}^1)$ which does not act proximally on the circle. Let us explain this in details.

First, one can reduce to the case where the action is minimal. Since $G$ has no finite orbit, there is a unique minimal $G$-invariant closed subset of the circle, that we denote by $\mathcal M$, see~\cite[p. 351]{Ghys}. The set $\mathcal M$ is either the whole circle or a Cantor set. 
Elementary arguments involving the maximum principle for $\mu$-harmonic functions, see~\cite[Lemme 5.1]{DKN},
show that the support of any stationary measure $\nu$ on the circle is $\mathcal M$, and that $\nu$ has no atoms. The map $s : x\in {\bf S}^1 \mapsto \nu([0,x]) \in {\bf S}^1$ 
is a non-decreasing continuous map of degree $1$ from the circle to itself. Moreover, it can be shown that there is an action $m: G \times {\bf S}^1 \rightarrow {\bf S}^1 $ of $G$ on the circle, such that $s\circ g = m(g) \circ s$ for every $g$ in $G$, see~\cite[p. 353]{Ghys}. By construction, this new action $m$ is minimal, i.e. every orbit is dense, and because every preimage of $s$ but a countable number is a point, $m$ does not preserve a probability measure on ${\bf S}^1$. 

The second step reduces to the case where the action is proximal. This is an elegant argument of Ghys showing that a non elementary minimal circle action is a finite cover of a proximal circle action, see~\cite[p. 362] {Ghys}. More precisely, suppose that there is a finite order homeomorphism $r : {\bf S}^1 \rightarrow {\bf S}^1$ which commutes to $m(g)$ for every $g$. Observe that a finite order homeomorphism of the circle is topologically conjugated to a rotation, hence the quotient $\pi : {\bf S}^1 \rightarrow  {\bf S} = {\bf S}^1/ r$ is a covering from the circle to a topological space homeomorphic to the circle. Because $r$ commutes with the action $m$, there is an action $m'$ of $G$ on ${\bf S}$ which is such that $\pi \circ m (g) = m'(g) \circ \pi$ for every $g$. Ghys's argument shows that if $m$ is a minimal action of $G$ on ${\bf S}^1$, there is a finite order homeomorphism $r$ such that the induced action $m'$ of $G$ on $\bf S$ is proximal.  

Thus, starting from any action on the circle which does not preserve a probability measure, one may first semi-conjugate 
it to a minimal action, and then consider a finite equivariant quotient 
via a finite covering of the circle that provides a proximal action, and hence a $(G,\mu)$-boundary. 
This quotient will be denoted by $p:{\bf S}^1\rightarrow {\bf S}$, 
(i.e. $p$ is the composition of the semiconjugation $s$ with the finite covering $\pi$), 
and the image of the stationary measure $\nu$ by $p$ will be denoted by $\nu'$. This quotient is finite in the following sense: $\nu'$-almost every point of ${\bf S}$ has $d$ preimages by $p$, where $d$ is the degree of the finite covering $\pi$, or equivalently the order of $r$.  

\subsection{Statement of the result}

For $\rho >0$, let 
\[ A_{\rho} := \{  z\in {\bf C}/ {\bf Z}: \ -\rho < \mathrm{Im} (z) < \rho\} . \]
If $g$ is an analytic diffeomorphism of ${\bf R} / {\bf Z}$, then it extends analytically to an injective holomorphic map $\widetilde{g} : A_{\rho} \rightarrow {\bf R}/{\bf Z}$ for some $\rho >0$. We denote by $\rho(g)$ the supremum of these numbers.

The main result of this note can be stated as follows:

\begin{theorem} \label{T: poisson boundary}
Let $G$ be a countable group of diffeomorphisms of the circle of class $C^{1+\mathrm{holder}}$ and 
$\mu$ be a probability measure on $G$ of finite entropy $-\sum_g \mu(g) \log \mu(g) <\infty$ such that 
\begin{equation} 
\label{eq: holder moment condition}
\int |\log g' |_{\tau} d\mu (g) <\infty,
\end{equation}
for some $\tau >0$, where $|\varphi|_{\tau}= \sup _{x,y} \frac{|\varphi(y) -\varphi(x)|}{|y-x|^{\tau}}$.
Suppose that the support of $\mu$ generates $G$ as a semi-group, and that $G$ does not preserve 
any probability measure on the circle. Then, the $(G,\mu)$-boundary $({\bf S}, \nu)$ is the Poisson boundary
in the following situations: 
\begin{itemize}
\item the action is $C^{1+\mathrm{holder}}$ and strongly locally discrete in the $C^1$-topology.
\item the action is $C^2$ and strongly locally discrete in the $C^2$-topology, and 
\begin{equation}\label{eq: logarithmic moment condition}  \int  |Lg|_{\infty} d\mu(g) <\infty,\end{equation}
where $L g = \frac{g''}{g'} $ is the logarithmic derivative of $g$. 
\item the action is $C^3$ and strongly locally discrete in the $C^3$-topology, and 
\begin{equation}\label{eq: schwarzian moment condition}  \int  |Sg|_{\infty} d\mu(g) <\infty,\end{equation}
where $Sg = (\frac{g''}{g'})' - \frac{1}{2} (\frac{g''}{g'})^2 $ is the Schwarzian derivative of $g$. 
\item the action is analytic and locally discrete in the $C^{\omega}$-topology, and 
\[ \int \frac{1}{\rho (g)} d\mu(g) < \infty \ \ \mathrm{and} \ \ \int |(\log g')' |_{\infty, A_{\rho(g)/2}} d\mu(g) < \infty , \] 
\end{itemize}  
\end{theorem} 

The proofs of the theorem will be done in the case where $G$ is a subgroup of preserving orientation diffeomorphisms. This does not affect the result since if $H\subset G$ is a finite index subgroup of $G$, and if $\mu_H$ is the balayage of $\mu$ on $H$, then the Poisson boundary of $(G,\mu)$ is the Poisson boundary of $(H,\mu_H)$. 

\subsection{Acknowledgments.} 
This work is a continuation of our joint work~\cite{DKN} in collaboration with Victor Kleptsyn and Andr\'es Navas. The problem of determining the Poisson boundary of groups of circle diffeomorphisms was raised during our common visit to the IHES in 2006. I would like to thank this institution for the very nice working conditions that was offered to us. 

It is a pleasure to thank the department of mathematics of Tokyo University, for his warm hospitality during the period this work was written, in december 2008, and for letting me the opportunity to give lectures on this subject. 

I also thank Andr\'es Navas and the referee for their careful reading. 

\section{Entropy criterion} \label{S: entropy criterion}

In the case where the entropy $H(\mu)= -\sum_g \mu(g) \log\mu(g)$ of $\mu$ is finite, there is a useful criterion 
to determine whether a $(G,\mu)$-boundary is the Poisson boundary. Recall that the asymptotic entropy of the pair 
$(G,\mu)$ is defined by 
\[  h(G,\mu) = \lim _n \frac{1}{n} H(\mu^{*n}) \]
which converges because the sequence $H(\mu^{*n})$ is sub-additive, see~\cite{Avez}. 
The significance of the asymptotic entropy is given by the Shannon-Breiman-McMillan theorem 
which says that for $\mu^{\bf N}$-a.e. ${\bf g}\in G^{\bf N}$, the sequence 
$-\frac{1}{n}\log \mu^{*n} ({\bf r_n})$ converges to $h(G,\mu)$ in $L^1(\mu^{\bf N})$. Thus, 
if $E_n$ is a sequence of sets contained in $G$ with $\mu^{*n}(E_n)\geq 1/2$, then 
\[  \liminf_n \frac{\log |E_n|}{n} \geq h(G,\mu),\]
where $|E|$ denotes the cardinal of a set $E$.  

Furstenberg introduced the real number 
\[  h_{\nu'} := - \int_{\bf S} \log \frac{d g^{-1} \nu '} {d\nu' } (x) d\mu(g)d\nu'(x) ,\]
called the \textit{boundary entropy}. The inequalities $0\leq h_{\nu'} \leq h(G,\mu)$ follows from the fact that the logarithm is concave. Independently, Derriennic~\cite{Derriennic} and Kaimanovich-Vershik~\cite{KV} showed that 
the $(G,\mu)$-boundary $({\bf S}, \nu')$ is the Poisson boundary if and only if $h_{\nu'}= h(G,\mu)$.

The boundary entropy $h_{\nu'}$ has a dynamical significance. To see this, it will be convenient  
to introduce the number 
\[  h_{\nu} = - \int _{{\bf S} ^1} \log \frac{d g^{-1} \nu} {d\nu} (x) d\mu(g)d\nu(x) ,\]
and to observe that $h_{\nu} = h_{\nu'}$. This is indeed a consequence of the fact 
that the stationary measure on the circle ${\bf S}^1$ after semi-conjugation is unique, and hence invariant 
by the finite order rotation $r$ that defines the covering $\pi$, see~\cite[Proposition 5.5]{DKN}. 
Consider the left random walk on $G$ defined by 
\[  {\bf l_n} = g_n \ldots g_1.\]
Since for every integer $n$ and $\nu$-almost every $x$, we have 
\[ \log \frac {d {\bf l_n} ^{-1} \nu}{d\nu}(x) = \log \frac {d  g_1 ^{-1} \nu}{d\nu}(x) + \ldots + \log \frac {d g_n ^{-1} \nu}{d\nu}(l_{n-1}(x)) \]
the random ergodic theorem shows that for $\mu^{\bf N}$-a.e. ${\bf g} \in G^{\bf N}$ we have 
\[  \lim_{n\rightarrow \infty} \frac{1}{n}\log \frac {d {\bf l_n} ^{-1} \nu}{d\nu}(x) = -h_{\nu},\]
the convergence being in $L^1(\nu)$. 
Thus, if $J$ is an interval with $\nu(J) >0$, we get almost surely 
\[ \liminf_n \log \frac{\nu({\bf l_n} J)} {\nu(J)} = \liminf_n \log \int_J \frac{d {\bf l_n}^{-1} \nu}{d\nu}(x) \frac{d\nu(x)}{\nu(J)} 
\geq \int_J \lim \log \frac{d {\bf l_n}^{-1} \nu}{d\nu}(x) \frac{d\nu(x)}{\nu(J)} = -h_{\nu}.\]
If $\varepsilon >0$, then for every ${\bf g}\in G^{\bf N}$ we define the number 
\begin{equation} 
\label{eq: constant C1} 
C_1(\varepsilon, J, {\bf g}) := \inf _{n\geq 0} \big( \nu({\bf l_n} J) \exp ((h_{\nu}+\varepsilon) n) \big),
\end{equation} 
which is positive a.s.

\section{Lyapunov exponent and affine distortion: proof of the first part of the theorem} \label{S: first part of the theorem}

Let $G$ be a countable group of diffeomorphisms of the circle of class 
$C^{1}$ preserving the orientation. Suppose that there is no $G$-invariant probability measure 
on the circle. Let $\mu$ be a probability measure on $G$ whose support generates 
$G$ as a semi-group, and such that
\[\int |\log g'|_{\infty} d\mu(g)<\infty.\] 
By a theorem of Baxendale~\cite{Baxendale}, there exists a 
stationary measure $\nu$ such that the sum of the Lyapunov exponents is negative. The circle being of dimension $1$, there is only one Lyapunov exponent, whose expression is 
\[  \lambda(\nu) := \int \log g'(x) d\mu(g)d\nu(x).\]
Moreover, the stationary measure $\nu$ is unique, hence we have $\lambda := \lambda(\nu) <0$. An alternative proof of this fact can be found in~\cite[Proposition 5.9]{DKN} in the case where $\mu$ is symmetric.

The random ergodic theorem asserts that for $\nu$-a.e. $x$, 
and $\mu ^{\bf N}$-a.e. ${\bf g} = (g_n)_n \in G^{\bf N}$, we have 
\begin{equation} \label{eq: lyapunov exponent}
 \lim _{n\rightarrow \infty} \frac{1}{n} \log {\bf l_n} ' (x) = \lambda ,
\end{equation}
where ${\bf l_n} := g_n \ldots g_1$. In fact,  
equation (\ref{eq: lyapunov exponent}) holds for \textit{every} $x$ and $\mu^{\bf N}$-a.e. ${\bf g}\in G^{\bf N}$
(the reader can deduce these facts by the technique of distortions that will be explained hereafter; however we will not 
emphasize on this. See~\cite{DK,DKN} for precisions on these facts). Then, for any $x\in {\bf S}^1$, and $\mu^{\bf N}$-a.e. ${\bf g} \in G^{\bf N}$, 
there is a constant $C\geq 1$ such that 
\begin{equation}\label{eq: constant C2}  
\forall n\geq 0:\ \ \ \ \ \ \ \ \frac{1}{C}\exp( 3 n \lambda /2) \leq {\bf l_n} '(x) \leq C\exp(n\lambda/2).
\end{equation}
We will denote by $C_2(x,{\bf g})\geq 1$ the infimum of the numbers $C$ verifying (\ref{eq: constant C2}). 

If $I$ is an interval of the circle, and $g$ is an element of $G$, we define the \textit{affine distortion} 
of $g$ in $I$ by 
\begin{equation}\label{def: distortion}  \kappa(g,I) := \sup _{x,y\in I} \log (\frac{g'(y)}{g'(x)} ).\end{equation}
The following lemma asserts that if $\tau>0$ and $G$ acts by $C^{1+\tau}$-diffeomorphisms, we have a control on the distortion 
of the maps ${\bf l_n}$ in a neighborhood of $x$, whose size is independent of~$n$. 

\begin{lemma} \label{L: affine distortion}
Let $\tau>0$. Suppose that $G$ acts by $C^{1+\tau}$ diffeomorphisms on the circle, without preserving any probability measure, and 
furthermore that the following moment condition is satisfied: 
\begin{equation}\label{eq: moment condition}  
\int |\log g'|_{\tau} d\mu (g) <\infty ,
\end{equation} 
where $|\varphi|_{\tau} = \sup _{x,y} \frac{|\varphi(x)-\varphi(y)|}{|x-y|^{\tau}}$. Then a.s. 
\[ C_3({\bf g}) := \sum_n  |\log g'_{n+1}|_{\tau} \exp(n\lambda \tau / 2) <\infty.\]
Let $\kappa >0$ be a number and let $r= r(\kappa,x,{\bf g}) >0$ be the number defined by the equation  
\[ r=\frac{\kappa^{1/\tau}e^{-\kappa}}{C_2(x,{\bf g}) C_3({\bf g})^{1/\tau}},\]
where $C_2(x,{\bf g})$ is defined in (\ref{eq: constant C2}).
For every $x\in {\bf S}^1$, $\mu^{\bf N}$-a.e. ${\bf g}=(g_n)\in G^{\bf N}$, and every integer $n\geq 0$,  
\[  \kappa ({\bf l_n}, [x-r,x+r])\leq \kappa .\]
\end{lemma}

\begin{proof} 
By the moment condition (\ref{eq: moment condition}), we have 
\[ \int C_3({\bf g}) d\mu^{\bf N} ({\bf g}) = \sum_n \exp(n\lambda \tau/2) \int |\log g'|_{\tau} d\mu(g) <\infty,\]
so that for $\mu^{\bf N}$-a.e. ${\bf g}\in G^{\bf N}$, $C_3({\bf g})$ is finite. This proves the first part of the lemma.

We denote $I=[x-r,x+r]$. 
We prove by induction that $\kappa({\bf l_n}, I) \leq \kappa$ for any $n\geq 0$ with the convention $l_0= id$. This is clear for $n=0$.
Suppose this is true for $k=0,\ldots, n$. For every $k\geq 0$, denote $x_k={\bf l}_k(x)$ and $y_k={\bf l_k}(y)$ if $y\in I$. We have for every $y\in I$
\[ \frac{{\bf l_{n+1}}'(y)} {{\bf l_{n+1}}'(x)} = \prod_{0\leq k\leq n} \frac {g'_{k+1} (y_k)}{g'_{k+1}(x_k)} \]
For every $k=0,\ldots,n$, because $\kappa({\bf l_k}, [x-r,x+r])\leq \kappa$,  
\[  |y_k - x_k| \leq  |\int _x ^y {\bf l_k}' (u) du | \leq r e^{\kappa} C_2(x,{\bf g}) \exp(k\lambda/2).\]  
Thus, 
\[ |\log  \frac{{\bf l_{n+1}}'(y)} {{\bf l_{n+1}}'(x)}| \leq \sum_{0\leq k\leq n} |\log g_{k+1}'|_{\tau} \cdot |y_k - x_k |^{\tau} 
\leq (re^{\kappa} C_2(x,{\bf g}) )^{\tau} C_3({\bf g}) \leq \kappa,\]
which completes the proof of the lemma.   
\end{proof}  

We will now prove that if $h_{\nu} < h(G,\mu)$, then the group $G$ 
is not strongly locally discrete in the $C^1$-topology. This and the entropy criterion (see section~\ref{S: entropy criterion}) will establish the first part of the 
theorem. We fix a number 
\[ 0<\varepsilon < h(G,\mu) - h_{\nu}.\]  

If the group $G$ does not preserve a probability measure on the circle, then 
there exists a point $x$ in the limit set 
and an element $l \in G$ such that $l(x)=x$ and $\alpha = l'(x) <1$, see \cite[Th\'eor\`eme F]{DKN} or \cite{N}. 
By the $C^{1+\tau}$-version of Sternberg linearization theorem~\cite{Chaperon},  
there is a germ of diffeomorphism $\varphi : ({\bf S}^1,x)\rightarrow ({\bf R}, 0)$ of class $C^{1+\tau}$ 
such that $\varphi \circ l  = l'(x) \varphi$. Thus by changing coordinates and keeping the fact that the action is 
$C^{1+\tau}$, we can suppose that the element $l\in G$ fixes $0\in {\bf S}^1$, and that there exists a positive number $\eta >0$ 
such that for every $y\in [-\eta,\eta]$, 
\[  l(y) = \alpha y.\]
We define $I:=[-\eta,\eta]$, and for every integer $m\geq 0$, $I_m : = l^m (I) $.

\begin{lemma} \label{L: construction of a sequence}
There exists an integer $m_0 \geq 0$ such that for every $m\geq m_0$, there exist two distinct elements $g_m$ and $h_m$ of $G$ such that  
\begin{enumerate} 
\item $\kappa (g_m , I)$ and $\kappa(h_m, I)$ tend to $0$ when $m$ tends to infinity,
\item for every $m\geq m_0$, $\lvert \log \frac{g_m '(0)}{h_m'(0)} \rvert \leq \frac{1}{m}$, 
\item for every $m\geq m_0$, the intervals $g_m(I_m)$ and $h_m(I_m)$ intersect. 
\end{enumerate}   
\end{lemma} 

\begin{proof} Fix some constants $C_1, C_2,C_3>0$ such that there is a set $\mathcal G_m \subset G^{\bf N}$
of positive measure $\mu^{\bf N}(\mathcal G_m)>0$ with the following property: 
for every ${\bf g}\in \mathcal G_m$, we have 
\[ C_1(\varepsilon, I_{2m} ,{\bf g}) \geq C_1,\ \ C_2(x,{\bf g})\leq C_2,\ \ C_3({\bf g})\leq C_3. \] 
The elements $g_m$ and $h_m$ will be of the form 
\begin{equation} \label{eq: definition} g_m = \overline{g_m} \circ l^m\ \ \ \mathrm{and}\ \ \ h_m = \overline{h_m} \circ l^m ,\end{equation} 
where $\overline{g_m}$ and $\overline{h_m}$ are chosen in ${\bf l_n}(\mathcal G_m)$. 

Denote by $|I_m|$ the length of $I_m$, and choose an integer $m_0$ such that $|I_m|^{\tau} C_2^{\tau} C_3 < 1/\tau e$ for $m\geq m_0$. Then the equation  
\begin{equation} 
\label{eq: definition of kappa}
\kappa_m ^{1/\tau} \exp(-\kappa_m) = |I_m| C_2 C_3^{1/\tau} 
\end{equation}
has two positive solutions; let $\kappa_m>0$ be the smallest of these solutions. Observe that the sequence $\kappa_m$ converges to $0$ when $m$ tends to infinity, since 
$C_2$ and $C_3$ do not depend on $m$, and the length of $I_m$ tends to $0$. By lemma~\ref{L: affine distortion}, any pair of elements $\overline{g_m}$ and $\overline{h_m}$ in ${\bf l_n} (\mathcal G_m)$ have a controlled distortion on $I_{m}$, namely
\[ \kappa (\overline{g_m}, I_m) \leq \kappa_m,\ \ \kappa (\overline{h_m}, I_m) \leq \kappa_m .\]  
Because $l^m$ is an affine map on $I$, we deduce 
\begin{equation}\label{eq: distortion condition} \kappa (g_m, I) \leq \kappa_m,\ \ \kappa (h_m, I) \leq \kappa_m . \end{equation}
Thus the condition 1) is automatically satisfied if $g_m$ and $h_m$ are chosen as in~\eqref{eq: definition}. 

Let us now indicate how to construct $\overline{g_m}$ and $\overline{h_m}$ so that the conditions 2) and 3) are also satisfied. For every ${\bf g} \in \mathcal G_m$, and every $n\geq 0$, we have 
\[ \frac{3\lambda n}{2} -\log C_2\leq \log {\bf l_n} '(0) \leq \frac{\lambda n }{2} +\log C_2 .\] 
Let us consider a covering of the interval $[\frac{3\lambda n}{2} -\log C_2, \frac{\lambda n }{2} +\log C_2]$ by at most \begin{equation} \label{eq: number of intervals} N = m \lceil |\lambda| n + 2\log C_2 \rceil \end{equation} 
intervals $I_{1},\ldots, I_N$ of length $1/m$. For every $k=1,\ldots , N$, let $G_{n,k}$ be the set of elements $g$ of ${\bf l_n} (\mathcal G_m)$ such that $\log g' (0)$ belongs to $I_k$. Then 
\[ {\bf l_n} (\mathcal G_m) = \bigcup _{k=1,\ldots, N} G_{n,k} ,\]
and if $k_n$ is the index such that the cardinal $|G_{n,k_n}|$ of $G_{n,k_n}$ is maximal, then 
\[ |G_{n,k_n} | \geq \frac{|{\bf l_n}(\mathcal G_m)|}{N}. \]
Since the measure of $\mathcal G_m$ is positive, the Shannon-Breiman-McMillan theorem shows that 
\[ \liminf_{n\rightarrow \infty} \frac{\log | {\bf l_n}(\mathcal G_m)| }{ n} \geq h(G,\mu),  \]
By \eqref{eq: number of intervals}, the number $N$ depends lineally of $n$, so that we get 
\begin{equation} \label{eq: entropy estimates}  \liminf_{n\rightarrow \infty} \frac{\log |G_{n,k_n}| }{ n} \geq h(G,\mu) .\end{equation}
Now since $G_{n,k_n} \subset {\bf l_n} (\mathcal G_m)$ and that $C_1(\varepsilon, I_{2m}, {\bf g})  \geq C_1$ for every ${\bf g}\in \mathcal G_m$, we have 
\begin{equation} \label{eq: measure estimates} \nu ( g ( I_{2m} ) ) \geq C_1 \exp ( -(h_{\nu} + \varepsilon ) n) \end{equation}
for every $g\in G_{n,k_n}$. Because $h_{\nu} + \varepsilon < h(G,\mu)$, we deduce from~\eqref{eq: entropy estimates} and \eqref{eq: measure estimates} that 
\[ \sum _{g\in G_{n,k_n} } \nu ( g( I_{2m} )) >1, \]
if $n$ is sufficiently large. For such values of $n$, the pigeon hole principle shows that there exist two distinct elements $\overline{g_m}$ and $\overline{h_m}$ of $G_{n,k_n}$ such that 
$$ \overline{g_m} ( I_{2m} )\cap \overline{h_m} (I_{2m} ) \neq \emptyset .$$ 
Observe also that the numbers $\log \overline{g_m}'(0)$ and $\log \overline{h_m}'(0)$ both belong to the interval $I_{k_n}$, whose length is  less than $1/m$, so that we have 
\[ \lvert \log \frac{\overline{g_m}'(0)}{\overline{h_m}'(0)} \rvert \leq \frac{1}{m}.  \]
Hence, by setting as before $g_m := \overline{g_m} \circ l^m$ and $h_m := \overline{h_m} \circ l^m$, we have
\[ g_m ( I_{m} )\cap h_m (I_{m} ) = \overline{g_m} ( I_{2m} )\cap \overline{h_m} (I_{2m} ) \neq \emptyset , \]
and
\[ \lvert \log \frac{g_m'(0) }{h_m'(0)} \rvert =  |\log \frac{\overline{g_m}'(0)}{\overline{h_m}'(0)} | \leq \frac{1}{m}.  \]
Together with \eqref{eq: distortion condition}, this establishes the lemma. 
\end{proof}

We claim that if $(g_m)_{m\geq m_0}$ and $(h_m)_{m\geq m_0}$ are the sequences given by lemma~\ref{L: construction of a sequence}, then when $m$ tends to infinity, $h_m^{-1} \circ g_m$ converges to the identity in the $C^1$-topology on every compact subset contained in the interior of $I$, in particular on $[-\eta/2, \eta /2]$.  Since every element $h_m^{-1}\circ g_m$ is not the identity, this will prove that the group $G$ is not strongly locally discrete, and will conclude the proof of the first part of theorem~\ref{T: poisson boundary}.

Observe first that for every $x,y\in I$ such that $x<y$, we have the inequalities 
\begin{equation}\label{eq: distortion}  \exp (-\kappa_m ) g_m '(0) (y-x) \leq g_m (y) - g_m (x) \leq \exp (\kappa_m ) g_m '(0) (y-x).  \end{equation}
Choose two points $u\in I_m$ and $v\in I_m$ such that $g_m (u)$ and $h_m (v) $ are equal and denote by $z$ their common value. We get 
\[ g_m (\eta ) \geq z + \exp (-\kappa_m ) g_m '(0) (\eta - u) \geq z + \exp(-\kappa_m) g_m '(0) \eta (1-\alpha^m). \] 
Similarly, we have 
\[ g_m (-\eta) \leq z - \exp(-\kappa_m) g_m '(0) \eta (1-\alpha^m),\]
and  
\[ h_m (\eta) \geq z + \exp(-\kappa_m) h_m '(0) \eta (1-\alpha^m) \ \ \mathrm{and}\ \  h_m (-\eta) \leq z - \exp(-\kappa_m) h_m '(0) \eta (1-\alpha^m). \]
Hence, if we denote by $M$ the minimal value of $g_m ' (0) $ and $h_m '(0)$, $h_m (I) \cap g_m (I)$ contains the interval 
\[ J = [z - \exp(-\kappa_m) M \eta (1-\alpha^m), z + \exp(-\kappa_m) M \eta (1-\alpha^m) ] . \]
Denote by $\alpha_m< \beta_m$ and $\gamma_m< \delta_m $ some elements of $I$ such that $g_m ([\alpha_m , \beta_m]) = h_m ([\gamma_m , \delta_m ])=J$. By \eqref{eq: distortion},  we have 
\[  g_m (\beta_m ) - g_m (\alpha_m) \leq (\beta_m - \alpha_m) g_m '(0) \exp (\kappa_m),\]
and hence 
\[ \beta_m -\alpha_m \geq 2\eta \big( \exp (-2\kappa_m) \frac{M}{g_m'(0)} (1-\alpha^m) \big) \geq 2\eta c_m, \]
with $c_m := \exp (-2\kappa_m - \frac{1}{m})  (1-\alpha^m)$.  
Similarly 
\[ \delta_m - \gamma_m \geq 2\eta c_m. \]
Because $I$ has length $2\eta$, and that $c_n$ tends to $1$ when $m$ tends to infinity, this implies that $\alpha_m$ and $\gamma_m$ tend to $-\eta$, and that $\beta_m$ and $\delta_m$ tend to $\eta$. To conclude, observe that the map $h_m^{-1} \circ g_m$ is a diffeomorphism from $[\alpha_m, \beta_m]$ to $[\gamma_n, \delta_m]$ whose derivative is approximately $1$ on $[\alpha_m, \beta_m]$ up to a multiplicative error of $\exp (2\kappa_m + \frac{1}{m})$. Hence, the map $h_m^{-1}  \circ g_m$ tends to the identity in the $C^1$-topology on the compact subsets of the interior of $I$. First part of theorem~\ref{T: poisson boundary} is proved. 

\section{Logarithmic and projective distortions: proof of the second and third part of the theorem}

Suppose that the group $G$ acts by diffeomorphisms of class $C^2$ (resp. $C^3$) on the circle. 
In this section, we prove that if $h_{\nu} < h(G,\mu)$, and if the moment condition~\eqref{eq: logarithmic moment condition} (resp.~\eqref{eq: schwarzian moment condition}) is satisfied, then the group is not strongly locally discrete in the $C^2$-topology (resp. $C^3$-topology). This establishes the second and third part of theorem~\ref{T: poisson boundary}, using the entropy criterion, see section~\ref{S: entropy criterion}.

\begin{lemma}
\label{L: control of Schwarzian derivative} 
Suppose that the following moment condition holds: 
\[ \int |Lg|_{\infty} d\mu (g) < \infty\ \ \ \ \ (\text{resp.}\ \ \int |Sg|_{\infty} d\mu(g) <\infty). \]  
Then, for ${\mu}^{\bf N}$-a.e. ${\bf g}\in G^{\bf N}$, the number 
\[  C_4({\bf g}) :=\sum_{n\geq 0} |L g_{n+1}|_{\infty} \exp(\lambda n/2)\ \ \ \ \ (\text{resp.}\ \  C_4({\bf g}):= \sum_{n\geq 0} |S g_{n+1}|_{\infty} \exp(\lambda n) ) \]
is finite. Moreover, let $\kappa >0$, and $r=r(\kappa,x,{\bf g})>0$ be the positive number defined in lemma~\ref{L: affine distortion} with $\tau =1$.
For every $x\in {\bf S}^1$, and $\mu^{\bf N}$-a.e. ${\bf g}\in G^{\bf N}$ we have for every $n\geq 0$ and every $y\in[x-r,x+r]$, 
\[ |L{\bf l_n}(y)| < C_2(x,{\bf g}) C_4({\bf g}) \exp (\kappa)\ \ \ \ \ (\text{resp.}\ \  |S{\bf l_n}(y)| <  C_2(x,{\bf g})^2 C_4({\bf g}) \exp(2\kappa) ) .\]  
\end{lemma} 

\begin{proof} By lemma~\ref{L: affine distortion}, for $\mu^{\bf N}$-a.e. ${\bf g}\in G^{\bf N}$, for every $y\in [x-r,x+r]$ and every integer $n\geq 0$, 
\[  {\bf l_n} '(y) \leq \exp (\kappa ) {\bf l_n} '(x) \leq C_2(x,{\bf g}) \exp(\kappa + n\lambda /2 ) .\]
Moreover, because 
\[ \int C_4({\bf g}) d\mu^{\bf N} = \big(\int |L g|_{\infty}d\mu(g) \big) \sum_n \exp(\lambda n/2)<\infty,\]
 for $\mu^{\bf N}$-a.e. ${\bf g}\in G^{\bf N}$, $C_4({\bf g})$ is finite. The same argument works in case of the Schwarzian derivative assumption. 
The lemma follows from the following relations satisfied by the logarithmic and Schwarzian derivatives of composition of maps:
\[ L {\bf l_n} =  L g_1 + {\bf l_1}' L g_2 \circ {\bf l_1} + \ldots + {\bf l_{n-1}}' L g_n \circ {\bf l_{n-1}} \]
and
\[ S {\bf l_n} =  S g_1 + ({\bf l_1}')^2 S g_2 \circ {\bf l_1} + \ldots + ({\bf l_{n-1}}')^2 Sg_n \circ {\bf l_{n-1}}.\] 
\end{proof} 

In the sequel, we use the same notations as the one used in section~\ref{S: first part of the theorem}.
First, recall that if $l$ is of class $C^2$ (resp. $C^3$), then the germ of diffeomorphism 
which conjugates $l$ to a linear map is also of class $C^2$ (resp. $C^3$), see~\cite{yoccoz} or~\cite[Theorem 3.6.2]{N}. Thus, we can suppose that $x=0$ and 
$l(y)=\alpha y$ for $y\in I$, the action being $C^2$ (resp. $C^3$). Then the following variation of lemma~\ref{L: construction of a sequence} holds: 

\begin{lemma}
There exists an integer $m_0 \geq 0$ such that for every $m\geq m_0$, there exist two distinct elements $g_m$ and $h_m$ of $G$ such that the conditions 1), 2) and 3) of lemma~\ref{L: construction of a sequence} holds, and such that, moreover, the logarithmic (resp. Schwarzian) derivative of $g_m$ and $h_m$ tends to $0$ uniformly on $I$ when $m$ tends to infinity. 
\end{lemma}

\begin{proof} 
The proof is essentially the same as the proof of lemma~\ref{L: construction of a sequence}. 
Modify the definition of the sets $\mathcal G_m$ in the following way: for any integer $m\geq 0$, 
we choose constants $C_1,C_2,C_3, C_4>0$ such that $\mathcal G_m$ is the set consisting of elements 
${\bf g}\in G^{\bf N}$ such that 
\[  C_1(\varepsilon, J_m, {\bf g})\geq C_1,\ \ C_2(x,{\bf g})\leq C_2,\ \ C_3({\bf g})\leq C_3\ \ \mathrm{and}\ \ C_4({\bf g})\leq C_4,\]
and $\mu^{\bf N}(\mathcal G_m)>0$. Then similarly as in~\ref{L: construction of a sequence}, we construct an integer $m_0$ and two sequences $(g_m)_{m\geq m_0}$ and   
$(h_m)_{m\geq m_0} $ of elements of $G$, such that the properties 1), 2) and 3) of lemma~\ref{L: construction of a sequence} are satisfied. 

We then verify that the logarithmic (resp. Schwarzian) derivatives of the maps $g_n$ and $h_m$ converges uniformly to $0$ on $I$. To do so, recall that $g_n = \overline{g_n} \circ l^m $ and $h_m= \overline{h_m} \circ l^m$ where $\overline{g_m}$ and $\overline{h_m}$ are elements of ${\bf l_n} (\mathcal G_m)$. In particular, lemma \ref{L: control of Schwarzian derivative} shows that for every $y\in [-r, r]$, we have 
\[ L \overline{g_m} (y) \leq C_2 C_4 \exp (\kappa_m ) \ \ \ \ \ (\text{resp.}\ \  S \overline{g_m} (y) \leq C_2 ^2 C_4 \exp (2\kappa_m)).  \]   
We deduce that if $m$ is sufficiently large so that $\alpha^m \eta\leq r$, then for every $y\in I$: 
\[  L g_m (y) = \alpha^m \big((L \overline{g_m})\circ l^m \big)(y) \leq \alpha^{m} C_2 C_4 \exp(\kappa_m)\rightarrow_{m\rightarrow \infty} 0\]
in the case condition~\eqref{eq: logarithmic moment condition} is satisfied, and 
\[  S g_m (y) = \alpha^{2m} \big((S \overline{g_m})\circ l^m \big)(y) \leq \alpha^{2m}  C_2^2 C_4 \exp(2\kappa_m)\rightarrow_{m\rightarrow \infty} 0,\]
in the case condition~\eqref{eq: schwarzian moment condition} is satisfied. 
The same reasoning applies to the sequence $(h_m)_{m\geq m_0}$. 
This ends the proof of the lemma. \end{proof}

We proved in section~\ref{S: first part of the theorem} that the sequence $\varphi_m= h_m^{-1} \circ g_m$ converges to the identity in the $C^1$-topology in restriction to $[-\eta/2,\eta/2]$, 
when $m$ tends to infinity. Moreover, we have for every $y\in [-\eta/2,\eta/2]$,  
\[ L \varphi_m (y) = L g_m (y) - \frac{(g_m)'(y)}{(h_m)'(h_m)^{-1}(x))} L h_m (g_m(x)) ,\]
and respectively 
\[  S \varphi_m (y) = S g_m (y) - \big(\frac{(g_m)'(y)}{(h_m)'(h_m)^{-1}(x))}\big)^2 S h_m (g_m(x)) , \]
which proves that $L\varphi_m$ (resp. $S \varphi_m$) converges uniformly to $0$ on $[-\eta/2,\eta/2]$. In the case where condition~\eqref{eq: logarithmic moment condition} is satisfied we deduce immediately that $\varphi_m''$ converges to $0$ uniformly on $[-\eta/2, \eta/2]$, and hence that the restriction of $\varphi_m$ to $[-\eta/2,\eta/2]$ converges to the identity in the $C^2$-topology. The second part of the theorem is thus proved. The third part of the theorem follows from the following

\begin{lemma} \label{L: convergence to identity}
Let $I\subset {\bf S}^1$ be a closed interval and $\varphi_m:I\rightarrow {\bf S}^1$ be a sequence of diffeomorphisms of class $C^3$ which converges to the identity
in the $C^1$-topology when $m$ tends to infinity, and such that $S \varphi_m$ tends uniformly to $0$. Then 
$\varphi_m$ converges to the identity in the $C^3$-topology.
\end{lemma}

\begin{proof} Write $I=[x_-,x_+]$. For every $m$ there exists a point $x_m\in I$ such that  
\[  \frac{\varphi_m''(x_m)}{\varphi_m'(x_m)} = \frac{1}{|I|} \int_I (\log \varphi_m ')'(u) du = \frac{1}{|I|} (\log \varphi_m'(x_+)  -\log \varphi_m'(x_-)),\]
so that because the derivative $\varphi_m'$ converges uniformly to the constant $1$, the sequence of numbers $\varphi_m''(x_m)$ converges to $0$. 

Let $A_m\in \mathrm{PGL}(2,{\bf R})$ be the M\"obius map with the same Taylor development as the map $\varphi_m$ 
at the point $x_m$, untill the second order. 
Because $A_m$ is determined by the three numbers $\varphi_m(x_m)$, $\varphi_m'(x_m)$ and $\varphi_m''(x_m)$, which tend respectively 
to $x_m$, $1$ and $0$, the matrix $A_m$ converge to the identity up to multiplication by a constant, hence the map $A_m$ tends to the identity in the $C^{\infty}$-topology on $I$. 

Consider the map $k_m (y) : = -x_m + A_m ^{-1} \circ \varphi_m ( x_m + y )$, defined for $y\in -x_m + I $. 
We have 
\[  k_m(0)=0,\ k_m'(0)=1\ \ \mathrm{and}\ \ k_m''(0) =0.\]
Because the Schwarzian 
derivative is invariant by composition on the right by translations, and on the left by projective transformations, 
we get for every $y\in -x_m  +I$
\[ S k_m (y) = S \varphi_m (x_m + y) \rightarrow _{m\rightarrow \infty} 0\]
uniformly. Consider the two solutions $u_m$ and $v_m$ of the following second order differential 
equation 
\[  \frac{d^2u}{dy^2} + \frac{S k_m}{2}  u=0,\]
with initial conditions 
\[ u_m(0)=0,\ (u_m)'(0)=1,\ v_m(0)=1,\ (v_m)'(0)=0.\]
The classical theory of the Schwarzian derivative tells us that 
\[  k_m = \frac{u_m}{v_m}.\]
But, by continuous dependance of the solutions of linear differential equations with respect to parameters, and since $S k_m$ tends uniformly to $0$,
$u_m$ converges uniformly to the function $y$, and $v_m$ converges uniformly 
to the function $1$, when $m$ tends to infinity. Then simple computations using the fact 
that $u_m' v_m - v_m' u_m =1$ show 
\[ k_m '(y) = \frac{1}{(v_m)^2},\ \ \ k_m'' = -2\frac{(v_m)'}{(v_m)^3},\ \ \ k_m''' = \frac{S k_m}{(v_m)^2}+ 6\frac{((v_m)')^2}{(v_m)^4} . \] 
But the function $(v_m)'$ converges uniformly to a $0$ since its derivative $-S k_m v_m$
tends to $0$ uniformly and $v_m'(0)= 0$.  This proves that $k_m$ tends to the identity in the $C^3$-topology, and hence so is $\varphi_m$. 
\end{proof} 

\section{The analytic case}

The proof of the fourth part of theorem~\ref{T: poisson boundary} is essentially the same as the proofs of the first and second parts. It uses variations in the complex domain of the techniques of affine distortions developed in section~\ref{S: first part of the theorem}. 

If $g$ is a diffeomorphism from an open subset $D\subset {\bf C}/{\bf Z} $ to an open subset contained in ${\bf C}/{\bf Z}$, we extend the definition~\eqref{def: distortion} to the complex case by 
$$\kappa ( g, D ) := \sup _{x,y\in D} \log \frac{|g'(y)|}{|g'(x)|}.$$ 
The following result is the main distortion estimates that will be needed to treat the analytic case.

\begin{lemma} 
Suppose that 
\[ \int \frac{1}{\rho (g)} d\mu(g) < \infty \ \ \mathrm{and} \ \ \int |(\log g')' |_{\infty, A_{\rho(g)/2}} d\mu(g) < \infty .  \]
For every ${\bf g} \in G^{\bf N}$, we denote 
\[ C_3 = C_3({\bf g}) := \sum _{n\geq 0} \exp (\lambda n /2) \cdot |(\log g_{n+1}')' |_{\infty, A_{\rho( g_{n+1} )/2}} \exp (\lambda n /2)\] 
and
\[ C_5 = C_5({\bf g}): = \inf _{n\geq 0} \rho( g_n ) \exp(-\lambda n /2) .\]
These numbers are finite for $\mu^{\bf N}$-a.e. ${\bf g}$. Let $\kappa >0$ be any number and define 
\[ r : = \min \big( \frac{C_5}{2 \exp (\kappa) C_2 } , \frac{\kappa}{2 \exp(\kappa) C_2 C_3 } \big). \]
Then for every $n\geq 0$ the map ${\bf l_n}$ is well-defined on $D(x,r )$ and its distortion is bounded by 
\begin{equation}\label{eq: control of distortion 1} \kappa ( {\bf l_n}, D(x,r) ) \leq \kappa . \end{equation}
\end{lemma} 

\begin{proof} The fact that $C_3({\bf g})$ and $C_5({\bf g})$ are finite almost everywhere comes from the fact that 
\[ \int C_3 ({\bf g}) d\mu^{\bf N} ({\bf g}) < \infty \ \ \mathrm{and} \ \ \int C_5 ({\bf g}) d\mu^{\bf N} ({\bf g}) < \infty.\]
The proof of the second part is by induction. Suppose that ${\bf l_n}$ is well-defined on $D(x,r)$ and that \eqref{eq: control of distortion 1} holds for every $n= 0,\ldots , N$. For $N=0$ this holds because ${\bf l_0}$ is the identity. Because the distortion of the map ${\bf l_N}$ is bounded by $\kappa$ on $D(x,r)$, and that $|{\bf l_N}' (x)| \leq C_2 \exp (\lambda N /2)$, the set ${\bf l_N} ( D(x,r) ) $ is contained in the annulus $A _{\rho}$ with $\rho =  r C_2 \exp ( \kappa + \lambda N/2 )$. Thus, because
\[  r C_2 \exp ( \kappa + \lambda N/2 ) \leq C_5 \exp ( \lambda N /2) , \]
the map ${\bf l_{N+1}}$ is well-defined on $D(x,r)$. For every $y\in D(0, r)$ we have  
\[ \log \frac{|{\bf l_{N+1}}'(y)|}{|{\bf l_{N+1}}'(x)|} = \sum _{n= 0}^{N} \log \frac{|g_{n+1} '(y_n)|}{|g_{n+1}'(x_n)|}  ,\]
where as above $x_n= {\bf l_n} (x)$ and $y_n = {\bf l_n} (y)$. By our assumptions, the distance between $y_n$ and $x_n$ is less than $\exp (\kappa) C_2 \exp (\lambda n /2) r$. Because $\exp (\kappa) C_2 \exp (\lambda n /2) r \leq C_5 \exp (\lambda n /2) /2$ we get  $|y_n - x_n| \leq \rho (g_n) /2$, and so  
\[ \log \frac{|{\bf l_{N+1}}'(y)|}{|{\bf l_{N+1}}'(x)|} \leq \exp (\kappa) C_2 C_3 r \leq \kappa /2 . \]
Thus $\kappa ({\bf l_{N+1} }, D(x,r) ) \leq \kappa$ and the lemma is proved by induction.  
\end{proof}

Then the proof follows exactly the same strategy as in section~\ref{S: first part of the theorem}. We fix a hyperbolic fixed point $x \in {\bf S}^1$. The map $l$ which fixes $x$ and has derivative $<1$ is conjugated by an analytic map to the linear map $y\mapsto \alpha y$, by Koenig's theorem. We consider the coordinate $y$ in the neighborhood of $x$. Then lemma~\ref{L: construction of a sequence} holds true if $I$ is repaced by $D(x, \eta)$ and $I_m$ by $D(x, \eta \alpha ^m)$. By a reasoning analogous to the one at the end of section~\ref{S: first part of the theorem}, we deduce that the map $h_m^{-1} \circ g_m$ is defined on $\overline{D(x,\eta /2)}$ for $m$ large enough, and tends to the identity in the uniform topology when $m$ tends to infinity. We leave the details to the reader. This ends the proof of theorem~\ref{T: poisson boundary}.

\section{Open questions} 

We finish by listing some open questions related to this work:

\begin{enumerate} 

\item Is there a statement analogous to theorem~\ref{T: poisson boundary} in every regularity, namely $C^k$ for every $k\geq 2$? 

\item Is there a converse to theorem~\ref{T: poisson boundary}? Namely, if $G$ a finitely generated non elementary subgroup of $\mathrm{Diff} ({\bf S}^1) $ which is not strongly locally discrete, then is it true that $h_{\nu} < h(\mu)$? For instance, is it true that the circle is a strict boundary of Thompson's group $G$? \footnote{This would be interesting to have a positive answer to this question since it would prove that for $l=3,\omega$, if $G$ is a finitely generated and non elementary subgroup of $\mathrm{Diff}^l({\bf S}^1)$, the property of being $C^1$-strongly locally discrete is equivalent to the property of being $C^3$-strongly locally discrete (or $C^{\omega}$-locally discrete if $l=\omega$).}

\item Is there an example of a non elementary finitely generated group of diffeomorphisms of the circle which is discrete but not locally discrete in some smooth topology? 

\item Is the circle (equipped with the Lebesgue measure) the Poisson boundary of the Malliavin brownian motion on the group $\mathrm{Diff}^{3/2} ({\bf S}^1)$, see~\cite{M}?

\end{enumerate}

\vspace{0.35cm}

Bertrand Deroin

Universit\'e Paris-Sud \& CNRS, Lab. de Math\'ematiques, B\^at 425

91405 Orsay Cedex, France

Bertrand.Deroin@math.u-psud.fr

\vspace{0.3cm}

\end{small}

\end{document}